\documentclass[10pt,reqno]{amsart}
\usepackage{amsmath, amsthm}

\makeatletter

\numberwithin{equation}{section}
\newtheorem{thm}{\indent\bf Theorem}[section]
\newtheorem{lemma} [thm] {\indent\bf Lemma}

\newtheorem{Rem}   [thm] {\indent\bf Remark}
\newtheorem{ex}   [thm] {\indent\bf Example}

\let\sst=\scriptscriptstyle 

\begin{document}

\title[Morrey type spaces and multiplication operator in Sobolev spaces]
{Morrey type spaces and multiplication operator in Sobolev spaces}
\author{A. Canale}
\address{Universit\`a degli Studi di Salerno, Dipartimento di Matematica,  
Via Giovanni Paolo II  n.132, 84084 FISCIANO (Sa), Italy.}
\email{acanale@unisa.it}
\author{C. Tarantino}
\address{Universit\`a degli Studi di Napoli {\it Federico II}, 
Dipartimento di Scienze Economiche e Statistiche,
Complesso Universitario Monte Sant'Angelo, Via Cintia, 80126, Napoli, Italy.}
\email{ctarant@unina.it}

\begin{abstract}

The paper deals with the operator $u\rightarrow gu$ 
defined in the Sobolev space $W^{r,p}(\Omega)$
and which takes values in $L^p(\Omega)$
when $\Omega$ is an unbounded open subset in $R^n$. 
The functions $g$ belong to wider 
spaces of $L^p$ connected with the Morrey type spaces.
$L^p$ estimates and compactness results are stated.

\end{abstract}

\maketitle
\thispagestyle{empty}

\bigskip

{\bf AMS Subject Classifications}:  35J25, 46E35

\bigskip

{\bf Key Words}: multiplication operator, Morrey type spaces, $L^p$ estimates,
compactness results.

\bigskip\bigskip

\section {Introduction} 

\bigskip 
Let $\Omega$ be an unbounded open subset in $R^n$. 

In literature there are different results about the study of 
{\it multiplication operator} for a suitable function 
$g:\;\Omega\rightarrow C$ 
\begin{equation}\label{multiplication}
u\longrightarrow gu,\
\end{equation}
as an operator defined in a Sobolev space (with or without weight) 
and which takes values in a $L^p(\Omega)$ space.
 
In $W^{1,p}_0 (\Omega)$ or in $W^{1,p}(\Omega)$ 
with $\Omega$ regular enough, 
reference results are 
some well-known inequalities which state the boundedness of 
(\ref{multiplication}): 
Hardy type inequalities (see 
H.Brezis \cite{2}, A.Kufner \cite{12}, J.Ne{\v c}as \cite{13}) 
when $g(x)$ is an appropriate power of the distance of $x$ from a subset 
of $\partial \Omega$, C.Fefferman inequality \cite{10} (see, e.g. 
F.Chiarenza-M.Franciosi \cite{9}, F.Chiarenza-M.Frasca \cite{10}) 
obtained when $g$ belongs to suitable Morrey spaces.

Our interest is to study {\it multiplication operator}
in the Sobolev space $W^{r,p}(\Omega)$, 
$r\in N$, \break
$1\le p<+\infty$ when $g$ belongs to suitable 
Morrey type spaces ${\mathcal M}^{p,s}$ introduced in \cite{8}. 

In the recent paper \cite{8} have been stated an embedding result
and a Fefferman type inequality. 

This paper, which can be considered as a continuation of the previous work,
deepens the study of such spaces and, in addition,  
introduces meaningful subspaces of these spaces
in which $L^p$ estimates and compactness results are stated.

As we noted in \cite{8}, one of the aspects of our interest 
lies in the fact that this type of inequalities are useful tools to prove a priori bounds
when studying elliptic equations. These estimates enable us to state existence and uniqueness results.
For applications in the study of the a priori bounds
see \cite{3}, \cite{4}, \cite{5}, \cite{6}, \cite{7}. The spaces considered in
some of these papers are connected to the spaces ${\mathcal M}^{p,s}$.

These spaces are wider than $L^p$ spaces, than classical Morrey space and 
are connected, for suitable values of $s$,
to the well known Morrey spaces defined 
when $\Omega$ is an unbounded open subset in $R^n$.

In Section 2 and Section 3 we complete the description of the spaces ${\mathcal M}^{p,s}$
with respect to the previous paper
analyzing their inclusion properties
and the relations between such spaces and the Morrey spaces introduced 
in bounded and unbounded domanis until now.

In Section 4 we state  $L^p$ estimates when the functions $g$ belong
to suitable subspaces ${\tilde {\mathcal M}}^{p,s}$ and ${\mathcal M}^{p,s}_0$ defined in Section 2.
We explain also the dependence of the constants in the $L^p$ bounds. To this aim it is necessary 
to introduce a kind of modulus of continuity of functions belonging to the Morrey type spaces.
In the process of approximation by functions more regular the dependence 
of the constants in the estimates plays a key role.

The compactness result for the multiplication operator is stated in Section 5.

In conclusion, one of the main aspects of the paper consists in the study of a class of spaces 
that generalizes the definition of classical Morrey space, $L^{p,n-sp}$, $s\le \frac{n}{p}$.

Furthemore we observe that the spaces ${\mathcal M}^{p, s-{\frac{n}{p}}}$, $ s\ge \frac{n}{p}$,
provides an intermediate space between $L^p(\Omega)$ and  $L^p_{loc}(\overline\Omega)$
in addition to the fact that includes $L^\infty(\Omega)$.

The estimates obtained when 
the multiplication operator uses functions belonging to Morrey spaces,
as we pointed out above, are used in the study of elliptic equations.

We emphasize that our results are obtained 
using tools different from
those used to get some similar estimates in type Morrey spaces in unbounded domains
and are of more general type
(cfr. \cite{6}, \cite{7} where one can also find some references).

\bigskip\bigskip

\section {Notations and Morrey type spaces} 

\bigskip 

Let $R^n$ be the $n$-dimensional real euclidean space. We set 
$$
B_r(x)=\{y\in R^n :\; |y-x|< r\},\quad B_r=B_r(0)\, 
\qquad \forall x\in R^n,\>\forall r\in R_+.
$$ 
For any $x\in R^n$, we call {\it 
open infinite cone} having vertex at $x$ every set of the type 
$$
\{x+\lambda (y-x):\,\lambda\in 
R_+,\,\,|y-z|<r\},
$$ 
where $r\in R_+$ and $z\in R^n$ are such that $|z-x|>r$. 

For all $\theta\in ]0,\pi/2[$ and for all $x\in R^n$ 
we denote by $C_\theta(x)$ an open infinite cone 
having vertex at $x$ and opening $\theta$. 

For a fixed $C_\theta(x)$, we set 
$$
C_\theta(x,h)=C_\theta(x)\cap B_h(x)\,,\qquad\forall h\in R_+.
$$ 

\noindent Let $\Omega$ be an open set in $R^n$. 
We denote by $\Gamma (\Omega,\theta,h)$ the family of open cones 
$C\subset\subset\Omega$ of opening $\theta$ and height $h$. 

We assume that the following hypothesis is satisfied: 

\begin{itemize}
\medskip

\item[$h_1$)]  There exists 
$\theta\in ]0,{\pi/ 2}[$ such that 
$$\forall x\in \Omega\qquad \exists C_{\theta}(x)\quad \hbox{such that} \quad 
\overline 
{C_{\theta}(x, \rho)}\subset \Omega.
$$ 
\end{itemize}
\medskip
\noindent
Let $(\Omega_\rho(x))_{x\in \Omega}$ be the family of open sets 
in $R^n$ defined as
$$\Omega_\rho(x)= B_\rho(x)\cap \Omega, 
\qquad x\in \Omega,
\qquad \rho>0.$$ 
If $1\le p< +\infty$ and $s\in R$, we denote by ${\mathcal M}^{p,s}(\Omega)$  
the space of functions $g\in L^p_{loc}(\overline\Omega)$ such that 
\begin{equation}\label{space}
\|g\|_{\mathcal M^{p,s}(\Omega)}=\sup_{x\in \Omega\atop \rho\in ]0,d]}\left(\rho^{s-n/p}\; 
\|g\|_{L^p(\Omega_\rho(x))}\right)<+\infty, 
\qquad d>0,
\end{equation} 
equipped with the norm defined by (\ref{space}). 
Furthermore let ${\tilde {{\mathcal M}}}^{p,s}(\Omega)$ be the closure
of $L^\infty(\Omega)$ in ${\mathcal M}^{p,s}(\Omega)$
and ${\mathcal M}^{p,s}_0(\Omega)$ the closure of
$C^\infty_0(\Omega)$ in ${\mathcal M}^{p,s}(\Omega)$.

\vfill\eject

\begin{Rem}
We observe that:

\begin{itemize}
\medskip

\item[1.]  If $\Omega$ is a bounded open set, 
$s\in\left[0,n/p\right]$  
and $d=\hbox{diam}\,\Omega$, 
then
${\mathcal M}^{p,s}(\Omega)$ is the space $L^{p,n-sp}(\Omega)$,
the classical Morrey space.

\medskip

\item[2.] If $\Omega=R^n$ then
$L^{p,n-sp}(\Omega)\subset {\mathcal M}^{p,s}(\Omega)$, 
$s\in\left[0,n/p\right]$.

\medskip 

\item[3.] The spaces ${\mathcal M}^{p,s}(\Omega)$
are riduced to the known Morrey space $M^{p,n-sp}(\Omega)$,
$s\in\left]0,n/p\right]$,
introduced in analogy to the classical Morrey spaces  
when $\Omega$ is an unbounded open set.

\medskip

\item[4.] For $s=0$, the norm (\ref{space}) is a sort of
average integral on $B_\rho(x)$. In particular this is true if 
the family  $\{\Omega_\rho (x)\}$ shrinks nicely to $x$, that is if
$\Omega_\rho(x)\subset B_\rho(x)$ for any $\rho > 0$ and
there is a constant $\alpha> 0$, independent of $\rho$, such that 
$|\Omega_\rho(x)|> \alpha |B_\rho(x)|$. 
The Theorem 3.1 in Section 3 states that 
$L^\infty(\Omega)$ is continuously embedded in ${\mathcal M}^{p,0}(\Omega)$. 

\medskip

\item[5.] If $s<0$, in analogy to the classical Morrey spaces, 
one can see that 
${\mathcal M}^{p,s}(\Omega)=\{0\}$ by
the Lebesgue differentiation Theorem. 
\end{itemize}
\end{Rem}

\bigskip 
For functions that belong to the spaces ${\tilde {\mathcal M}}^{p,s}(\Omega)$
we can state the following alternative definition.
Let $\Sigma(\Omega)$ the 
$\sigma$-algebra of Lebesgue measurable subsets of $\Omega$. 
\bigskip

\begin  {lemma}\label {sigma}

A function $g \in {\tilde {\mathcal M}}^{p,s}(\Omega)$ if 
and only if $g\in {\mathcal M}^{p,s}(\Omega)$ and the function 

\begin{equation}\label{sigma}
\sigma^{p,s}_g:\quad t\in [0,1]\longrightarrow 
\sup_{E\in \Sigma(\Omega)\atop 
\displaystyle\mathop{\hbox{sup}}_{x\atop {\rho\in]0,d] }}
\scriptstyle{\rho^{-n}|E\cap B_\rho (x)|\le t}}
\|g\chi_{\sst E}\|_{{\mathcal M}^{p,s}(\Omega)}
\end{equation}
is continuous in zero, 
where $\chi_{E}$ denotes the characteristic function of $E$.
\end{lemma}

\medskip

\begin{proof}

Let $g$ be a function belonging to the space ${\mathcal M}^{p,s}(\Omega)$ and let
$\sigma^{p,s}_g$ be continuous in zero. 
We denote by $\delta_\epsilon$ a positive number such that 

$$
E\in \Sigma(\Omega)\,,\quad 
\sup_{x\in \Omega\atop {\rho\in]0,d] }}\left( \rho^{-n}|E\cap B
_\rho (x)|\right)\le \delta_\epsilon 
\Longrightarrow \|g\chi_{\sst E}\|_{{\mathcal M}^{p,s}(\Omega)}\le \epsilon
$$ 
and by $\Omega_r$ the set

\begin{equation}\label{omega r}
\Omega_r(g)=\{x\in \Omega\,:\;|g(x)|\ge r\}\qquad\forall r\in R_+.
\end{equation}
Then there exists a positive constant $c\in R_+$, independent of $r$ and $g$, 
such that 

$$
\sup_{x\in \Omega\atop {\rho\in]0,d] }}\left( \rho^{-n}|\Omega_r\ \cap B_\rho(x)|\right) 
\le c {\frac{\|g\|^p_{{\mathcal M}^{p,s}}(\Omega)} {r^p}}.
$$ 
If we set
$$
r_\epsilon=c\,\left({\|g\|^p_{{\mathcal M}^{p,s}(\Omega)} 
\over \delta_\epsilon}\right)^{1/ p}\,,
$$
we get
$$
\sup_{x\in \Omega\atop {\rho\in]0,d] }}\left( \rho^{-n}|\Omega_{r_\epsilon}
\cap B_\rho(x)|\right)\le \delta_\epsilon
$$ 
and, as a consequence, 
$$
\|g\chi_{{\sst \Omega_{k_\epsilon}}}\|_{{\mathcal M}^{p,s}(\Omega)}\le \epsilon\,.
$$ 
Let us introduce now the function
$$
g_\epsilon=g-g\chi_{{\sst \Omega_{k_\epsilon}}}.
$$ 
Evidently $g_\epsilon \in L^\infty (\Omega)$ and 
$$
\|g-g_\epsilon\|_{{\mathcal M}^{p,s}(\Omega)}= 
\|g\chi_{{\sst \Omega_{k_\epsilon}}}\|_{{\mathcal M}^{p,s}(\Omega)}\le \epsilon\,.
$$ 

\medskip

\noindent Conversely, let us assume $g\in{\tilde {\mathcal M}}^{p,s}\Omega)$. 
Then there exists $g_\epsilon\in L^\infty(\Omega)$ such that 
$$
\|g-g_\epsilon\|_{{\mathcal M}^{p,s}\Omega)}\le \epsilon/2.
$$ 
So we get
\begin{equation}
\begin{split}
\|g\chi_{\sst E}\|_{{\mathcal M}^{p,s}(\Omega)}&\le 
\|(g-g_\epsilon)\chi_{\sst E}\|_{{\mathcal M}^{p,s}(\Omega)}+ 
\|g_\epsilon\chi_{\sst E}\|_{{\mathcal M}^{p,s}(\Omega)} \le
\\& 
\le\epsilon/2+\rho^s\|g_\epsilon\|_{L^\infty(\Omega)} \sup_{x\in 
\Omega\atop {\rho\in]0,d] }}\left( \rho^{-n}|E\cap B_\rho (x)|\right)^{1\over p} 
\end{split}
\end{equation}
from which we deduce that
$$\|g\chi_{\sst E}\|_{{\mathcal M}^{p,s}(\Omega)}\le \epsilon$$ 
for any $E\in \Sigma(\Omega)$ such that 
$$
\sup_{x\in\Omega\atop {\rho\in]0,d] }}\rho^{-n}|\Omega_\rho(x)\cap E|\le 
\left({\epsilon\over 
{2\,d^s\|g_\epsilon \|_{L^\infty(\Omega)}}}\right)^p
$$ 
and the lemma is proved.

\end{proof}

\bigskip\bigskip

\section {Inclusion properties} 

\bigskip 

The following Theorem states inclusion properties of the spaces ${\mathcal M}^{p,s}(\Omega)$.
Some of these, $a)$ and $b)$, introduced in the previous paper \cite{8}, will be proved again 
for sake of completeness showing explicitly the value of the constants in the estimates.
\bigskip

\begin{thm} \label{inclusions}
The following inclusions hold:

\begin{itemize}
\medskip

\item[$a)$]
$L^\infty(\Omega)\hookrightarrow {\mathcal M}^{p,s}(\Omega)$ 
$\qquad\forall p\in[1,+\infty[\quad$ and 
$\quad\forall s\ge 0$.  

\medskip

\item[$b)$]
$L^q(\Omega) 
\hookrightarrow {\mathcal M}^{q,s}(\Omega)
\qquad \forall s\ge \frac{n}{q},\qquad$
$1\le p\le q < +\infty$.

\medskip

\item[$c)$]
${\mathcal M}^{q,s}(\Omega)
\hookrightarrow {\mathcal M}^{p,s}(\Omega), 
\qquad 1\le p\le q < +\infty$.

\medskip

\item[$d)$]
${\mathcal M}^{q,\mu}(\Omega)\hookrightarrow {\mathcal M}^{p,\lambda}(\Omega)\qquad$
$\lambda, \mu>0\quad$
$\frac{\lambda-n}{p}\le \frac{\mu-n}{q},\quad$ 
$1\le p\le q < +\infty$.
\medskip

\item[$e)$]
$L^q(\Omega) 
\subset {\mathcal M}^{p,s}_0(\Omega)
\subset {\tilde {\mathcal M}}^{p,s}(\Omega) 
\qquad \forall s\ge \frac{n}{q}\qquad$
$1\le p\le q < +\infty$.

\medskip

\item[$f)$]
${\mathcal M}^{q,s}(\Omega) 
\subset {\tilde {{\mathcal M}}^{p,s}}(\Omega),
\qquad $
$1\le p< q < +\infty$.

\end{itemize}
\end{thm}

\vfill\eject

\begin{proof} $$ $$

\begin{itemize}
\medskip

\item[$a)$]
If $g\in L^\infty(\Omega)$ we get
 
\begin{equation}
\begin{split}
\|g\|_{{\mathcal M}^{p,s}(\Omega)}&= 
\sup_{x\in \Omega\atop \rho\in]0,d]}\left(\rho^{s-n/p}
\;\| g\|_{L^p(\Omega_\rho(x))}\right)\le
\\&
\le \| g\|_{L^{\infty}(\Omega)} 
\sup_{x\in\Omega\atop \rho+\in]0,d]}\rho^s 
\left(\rho^{-n/p} |\Omega_\rho(x)|^{1/p}\right) 
=c\|g\|_{L^\infty(\Omega)}
\end{split}
\end{equation}
where, if $\omega_n$ denotes the volume of the unit ball $B(0,1)$,
we get $c=\omega_n^{\frac{1}{p}}d^s$.

\bigskip

\item[$b)$-$c)$]
By H{\"o}lder inequality it follows 

\begin{equation}
\begin{split}
\| g\|_{{\mathcal M}^{p,s}(\Omega)} &\le
\sup_{x\in \Omega\atop \rho\in]0,d]}(\rho^{s-{n\over p}}
\|g\|_{L^q(\Omega_\rho(x))}|\Omega_\rho(x)|^{{1\over p}-{1\over q}})\le
\\&
\le c_1\sup_{x\in \Omega\atop \rho\in]0,d]}
\rho^{s-{n\over q}}\|g\|_{L^q(\Omega_\rho(x))}
\le c_2
\|g\|_{L^q(\Omega)}, 
\end{split}
\end{equation}
with $c_2=c_1 d^{s-\frac{n}{q}}=\omega_n^{{1\over p}-{1\over q}}
d^{s-\frac{n}{q}}$.

\bigskip

\item[$d)$]
As above by H{\"o}lder inequality, since 
$\lambda\le n\left(1-\frac{p}{q}\right)+\frac{\mu p}{q}$, 
it results

\begin{equation}
\begin{split}
\int_{\Omega_\rho(x)}|g|^p &\le
|\Omega_\rho(x)|^{1-\frac{p}{q}}
\left(\int_{\Omega_\rho(x)}|g|^q\right)^{\frac{p}{q}}\le
\\& \le
\omega_n^{1-\frac{p}{q}}
\rho^{n\left(1-\frac{p}{q}\right)+\mu\frac{p}{q}}
\left(\rho^{-\mu}\int_{\Omega_\rho(x)}|g|^q\right)^{\frac{p}{q}}\le
\\&
\le \omega_n^{1-\frac{p}{q}} \rho^\lambda
d^{n\left(1-\frac{p}{q}\right)+\mu\frac{p}{q}-\lambda}
\left(\rho^{-\mu}\int_{\Omega_\rho(x)}|g|^q\right)^{\frac{p}{q}}
\end{split}
\end{equation}
from which
$$
\left(\rho^{-\lambda}\int_{\Omega_\rho(x)}
|g|^p\right)^{\frac{1}{p}}\le
c \left(\rho^{-\mu}\int_{\Omega_\rho(x)}
|g|^q\right)^{\frac{1}{q}}
$$
with $c=\omega_n^{1-\frac{p}{q}}
d^{n\left(1-\frac{p}{q}\right)+\mu\frac{p}{q}-\lambda}$.
So we can deduce the assertion.

\bigskip

\item[$e)$]
Let $g\in L^q(\Omega)$. By density of $C^\infty_0(\Omega)$ 
in $L^p$ it follows that
there exists a function 
$\phi_\epsilon\in C^\infty_0(\Omega)$ such that
$$
\|g-\phi_\epsilon\|_{L^q(\Omega)}\le 
\frac{\epsilon}{\omega_n^{{1\over p}-{1\over q}}d^{s-\frac{n}{q}}}.
$$
Then we get, again by H{\"o}lder inequality,

\begin{equation}
\begin{split}
\|g-\phi_\epsilon\|_{{\mathcal M}^{p,s}(\Omega)}&= 
\sup_{x\in \Omega\atop \rho\in]0,d]}\rho^{s-n/p}
\left(\| g-\phi_\epsilon\|_{L^p(\Omega_\rho(x))}\right)\le
\\&
\le \sup_{x\in\Omega\atop \rho\in]0,d]}
 \rho^{s-n/p}\| g-\phi_\epsilon\|_{L^q(\Omega_\rho(x))}
|\Omega_\rho(x)|^{{1\over p}-{1\over q}})\le 
\\&
\le \omega_n^{{1\over p}-{1\over q}}
d^{s-\frac{n}{q}}\| g-\phi_\epsilon\|_{L^q(\Omega)}<\epsilon.
\end{split}
\end{equation}
\medskip
\noindent The second inclusion is obvious.

\bigskip

\item[$f)$]

We observe that, if $g\in {\mathcal M}^{q,s}(\Omega)$, from 
$c)$ we have $g\in {\mathcal M}^{p,s}(\Omega)$. Furthermore 

\begin{equation}
\begin{split}
\|g\chi_{\sst E}\|_{{\mathcal M}^{p,s}(\Omega)}&= 
\sup_{x\in \Omega\atop{\rho\in ]0,d]}} \rho^{s-\frac{n}{p}}
\|g\|_{L^p(\Omega_\rho(x)\cap E)}\leq
\\&
\le  \sup_{x\in \Omega\atop{\rho\in ]0,d]}}\rho^{s-\frac{n}{p}}
|\Omega_\rho(x)\cap E|^{{1\over p}-{1\over q}} 
\|g\|_{L^q(\Omega_\rho(x)\cap E)} \le
\\ &\le 
\| g\|_{{\mathcal M}^{q,s}(\Omega)}\sup_{x\in \Omega\atop{\rho\in ]0,d]}}
(\rho ^{-n}\; |\Omega_\rho(x)\cap E|)^{{1\over p}-{1\over q}},
\end{split}
\end{equation}
and we deduce that the function $\sigma^{p,s}_g$, defined by (\ref{sigma}), is 
continuous in zero. From Lemma \ref{sigma} it follows that 
$g\in {\tilde {\mathcal M}}^{p,s}(\Omega)$. 

\end{itemize}

\end{proof}

The following simple examples show that the inclusion
\begin{equation}\label{inclusion}
L^p(\Omega) 
\subset {\mathcal M}^{p,\frac{n}{p}}(\Omega).
\end{equation}
is a strict inclusion.

\bigskip

\begin{ex}

The constant functions belong to $S^{p,\frac{n}{p}}(\Omega)$ and
do not belong to $L^p(\Omega)$. 
\end{ex}

\bigskip

\begin{ex}

The function 
$\frac{1}{1+|x|^\alpha}$ belongs to $S^{p,\frac{n}{p}}$ for any $\alpha>0$
but does not belong to $L^p$ if $\alpha\in ]0,\frac{n}{p}]$. 
\end{ex}

\bigskip\bigskip

\section{$L^p$ estimates}

\bigskip

Let $r, p, q$ be real number with the condition 

$$
h_2)\quad r\in N,\qquad 
1\le p\le q<+\infty, \qquad q\ge 
{n\over r}, \qquad q>{n\over r} \quad 
\hbox{if}\quad {n\over r}=p>1 .
$$ 
The Theorem \ref{embedding theo} below was stated in \cite{8}. 
We will derive our $L^p$ estimates using this result.

The proof of the Theorem uses the following inequality (see again \cite{8}, Lemma 4.2)
which states a connection between the functions belonging to $L^1(\Omega)$ 
and the  functions in $L^1(\Omega_\rho(x))$:

\begin{equation}\label {L1}
c_1\|v\|_{L^1(\Omega)}\le 
\int_{\Omega} \rho^{-n}\|v\|_{L^1(\Omega_\rho(x))}dx\le c_2
\|v\|_{L^1(\Omega),}
\quad c_1, c_2\in R_+,
\quad\forall v\in L^1(\Omega).
\end{equation}

\bigskip

\begin{thm}\label {embedding theo}  
If  $h_1)$ and $h_2)$  hold, 
then for any $g\in {\mathcal M}^{q,\frac{s}{p}}(\Omega)$, $s\le p$, and 
for any $u\in W^{r,p}(\Omega)$ 
we get $gu\in L^p(\Omega)$ and 

\begin{equation}\label{bound 2} 
\|gu\|_{L^p(\Omega)}\le C\|g\|_{{\mathcal M}^{q,\frac{s}{p}}(\Omega)} 
\|u\|_{W^{r,p}(\Omega)} ,
\end{equation}
where the constant $c=c(p,q,r,n)$ is independent of $g$ and $u$. 
\end{thm}

\bigskip

The next two lemma state $L^p$ estimates.

\bigskip

\begin{lemma}\label{tilde 1}
If $h_1)$, $h_2)$ hold and 
$g\in {\tilde {\mathcal M}}^{q,\frac{s}{p}(\Omega)}$, $s\le p$, 
then for any $\epsilon\in {R_+}$ there 
exists $c(\epsilon)\in {R_+}$ such that 
$$
\|gu\|_{L^p(\Omega)}\le \epsilon\|u\|_{W^{r,p}(\Omega)}+
c(\epsilon) \|u\|_{L^p(\Omega)}
\qquad \forall u\in {W^{r,p}(\Omega)}.
$$
\end{lemma} 

\medskip

\begin{proof}
Let $\phi_\epsilon\in L^\infty(\Omega)$ such that 
\begin{equation}\label{phi}
\|g-\phi_\epsilon\|_{{\mathcal M}^{q,\frac{s}{p}}(\Omega)}
\le {\epsilon/c},
\end{equation}
where the constant $c$ is the constant in the bound (\ref{bound 2}). 
Then by Theorem 4.1

\begin{equation}
\begin{split}
\|gu\|_{L^p(\Omega)} &\le
\|(g-\phi_\epsilon)u\|_{L^p(\Omega)}+
\|\phi_\epsilon u\|_{L^p(\Omega)}\le
\\&
\le c\|(g-\phi_\epsilon)u\|_{{\mathcal M}^{q,\frac{s}{p}}(\Omega)}
\|u\|_{W^{r,p}(\Omega)}+
\|\phi_\epsilon\|_{L^\infty(\Omega)}
\|u\|_{L^p(\Omega)}\le
\\ &
\le\epsilon \|u\|_{W^{r,p}(\Omega)}+c(\epsilon)
\|u\|_{L^p(\Omega)}
\end{split}
\end{equation}
with $c(\epsilon)=
\|\phi_\epsilon\|_{L^\infty(\Omega)}$.

\end{proof}

\bigskip

\begin{lemma}
If $h_1)$, $h_2)$ hold and $g\in {\mathcal M}^{q,\frac{s}{p}}_0(\Omega)$, $s\le p$,
then for any 
$\epsilon\in {R_+}$ there exist $c(\epsilon)\in {R_+}$ and an open set 
$\Omega_\epsilon \subset \subset \Omega$ with cone property such that 

\begin{equation}\label{bound 3}
\|gu\|_{L^p(\Omega)}\le
\epsilon \|u\|_{W^{r,p}(\Omega)}+c(\epsilon) 
\|u\|_{L^p(\Omega_\epsilon)} 
\qquad \forall u\in W^{r,p}(\Omega).
\end{equation}
\end{lemma}

\medskip 

\begin{proof}

Let $\phi_\epsilon\in C^\infty_0(\Omega)$ be such that (\ref{phi}) holds.

Reasoning as in the proof of the Lemma 4.2 we get 

\begin{equation}\label{bound 4}
\|gu\|_{L^p(\Omega)}\le 
\epsilon \|u\|_{W^{r,p}(\Omega)}+
c(\epsilon) \|u\|_{p,supp \phi_\epsilon},
\end{equation} 
with $c(\epsilon)$ as in the Lemma 4.2.

Then let us fix $\theta\in ]0,{\pi/ 2}[$ and $h_\epsilon\in ]0,\hbox 
{dist}(\partial\Omega,\hbox {supp}\, \phi_\epsilon)/ 2[$. If we 
denote by $\Omega_\epsilon$ the open set of $R^n$ union of the cones $C\in 
\Gamma (\Omega,\theta,h_\epsilon)$ such that $C\cap \;\hbox{supp}\; 
g_\epsilon \ne \emptyset$, then (\ref{bound 3}) follows from (\ref{bound 4}). 

\end{proof}

Let us define the modulus of continuity of a function 
$g\in{\tilde {\mathcal M}^{p,s}(\Omega)}$. 
First let us introduce the function

$$
\tau^p_s[g](t)=\sup_{E\in \Sigma(\Omega)\atop 
\displaystyle\mathop{\hbox {sup}}_{x}
\scriptstyle{|E\cap B_d (x)|\le t}}
\,\|g\,\chi_E\|_{\mathcal M^{p,s}(\Omega)}\,, 
\qquad t\in R_+\,,
$$
where $\chi_E$ is the characteristic function of $E$.
It follows by Lemma \ref{sigma} that
that $g\in{\tilde {\mathcal M}}^{p,s}(\Omega)$ if and 
only if $g\in \mathcal M^{p,s}(\Omega)$ and
$$
\lim_{t\rightarrow 0}\tau^p_s[g](t)=0\,.
$$
We define the {\it modulus of continuity} of 
$g\in{\tilde{\mathcal  M}}^{p,s}(\Omega)$ 
as a function $\tau[g]: R_+ \rightarrow R_+$ satisfying

$$
\tau^p_s[g](t)\le \tau[g](t)\qquad 
\forall t\in R_+\,,\qquad \quad
\lim_{t\rightarrow 0}\tau\,[g](t)=0\,.
$$
If $g\in L^p_{loc}(\overline\Omega)$, we get

\begin{equation}\label{omega a zero}
\lim_{r\rightarrow +\infty}\sup_x\left|\Omega_r(g)\cap B_d(x)\right|=0\,,
\end{equation}
where $\Omega_r(g)$ is defined in (\ref{omega r}). 
Let us denote, for any $k\in R_+$, by $r_k=r_k(g)$ a real number such that

$$
\sup_x\left|\Omega_{r_k}(g)\cap B_d(x)\right|\le
{\frac{1}{k}}
$$
and by $r[g]$ the function
\begin{equation}\label{r}
r[g]:k\in R_+\rightarrow r[g](k)=r_k\in R_+\,.
\end{equation}
Now we state the following lemma in which
we emphasize the dependence of the constants in the final bound.

\bigskip

\begin{lemma}
In the same hypotheses of Lemma  \ref{tilde 1}
for any $k\in R_+$ we get
$$
\|g\,u\|_{L^p(\Omega)}\le C\,\tau[g]\left
({1\over {k}}\right)\,\|u\|_{W^{r,p}(\Omega)}
+r[g](k)\,\|u\|_{L^p(\Omega)}\qquad\forall\,u\in W^1(\Omega)\,,$$
where $C$ is the constant in (\ref{bound 2}), $\tau[g]$ is the modulus of 
continuity of $g$ in $ {\tilde {\mathcal M}}^{q,\frac{s}{p}}(\Omega)$ 
and $r[g]$ is the function defined by (\ref{r}).
\end{lemma}

\medskip

\begin{proof} Let
$$
g_k=(1-\chi_{\Omega_{r_k}})\,g\,.
$$
The function $g_k$ so defined belongs to the space $L^\infty(\Omega)$.
From Theorem \ref{embedding theo} we get

\begin{equation}
\begin{split}
\|g\,u\|_{L^p(\Omega)} &\le
\|(g-g_k)\,u\|_{L^p(\Omega)}+\|g_k\,u\|_{L^p(\Omega)}\le 
\\& \le
C\|g-g_k\|_{\mathcal M^{q,\frac{s}{p}}(\Omega)}\, \|u\|_{W^{r,p}(\Omega)}+
\|g_k\,u\|_{L^p(\Omega)}=
\\ & =
C\|g\chi_{\Omega_{r_k}}\|_{\mathcal M^{p,s}(\Omega)}\,\|u\|_{W^{r,p}(\Omega)}+
r[g](k)\|u\|_{L^p(\Omega)}\,.
\end{split}
\end{equation}
Taking in mind (\ref{omega a zero}) and modulus of continuity we deduce the result.

\end{proof}

\bigskip\bigskip

\section{Compactness result}

\bigskip

Now we state the compactness result.

\bigskip

\begin{thm}
If $h_1)$, $h_2)$ hold and $g\in 
{\mathcal M}^{q,\frac{s}{p}}_0(\Omega)$, $s\le p$, 
then the operator 
$$
u\in {W^{r,p}(\Omega)} 
\longrightarrow gu\in L^p(\Omega)
$$ 
is compact.
\end{thm} 

\begin{proof}
We remark that for any open set $\Omega'\subset \subset \Omega$ 
the operator 
$$
u\in {W^{r,p}(\Omega)} \longrightarrow u\vert_{\Omega'}\in 
W^{r,p}(\Omega')
$$ 
is linear and bounded. 

On the other hand, if $\Omega'$ verifies the cone property too, by 
Rellich-Kondrachov's theorem the operator 
$$
u\in W^{r,p}(\Omega') \longrightarrow u\in L^p(\Omega')
$$ 
is compact. 
Then also the operator
$$
u\in {W^{r,p}(\Omega)} \longrightarrow u\vert_{\Omega'}\in 
L^p(\Omega')
$$
is compact.
Therefore if $\{u_n\}\in W^{r,p}(\Omega')$ is a bounded sequence
there exists a subsequence $\{u_{n_k}\}$ converging to $u$ in
$L^p(\Omega')$.
By Lemma 4.3 we get
$$
\|g(u_{n_k}-u)\|_{L^p(\Omega)}\le 
\epsilon \|u_{n_k}-u\|_{W^{r,p}(\Omega)}+
c(\epsilon) 
\|u_{n_k}-u\|_{L^p(\Omega')}
\qquad \forall u\in W^{r,p}(\Omega)
$$
from which we can deduce the result.
\end{proof}

\bigskip

\begin{Rem} 
We remark that if $g\in {\mathcal M}^{q,s}(\Omega)$, $s\ge 0$,
the results stated in Section 4 and Section 5 are still avalaible.
In the case  $s> 0$ this is due to the inclusion properties (see Theorem \ref{inclusions}).
Then the bounds stated when $g\in {\mathcal M}^{q,\frac{s}{p}}(\Omega)$ are of more general type.
\end{Rem}

\bigskip\bigskip

\renewcommand\refname{{

\end{document}